\documentclass[9pt]{amsart}

\usepackage[T1]{fontenc}
\usepackage[utf8]{inputenc}
\usepackage{amsmath,amssymb,amsthm,mathtools}
\usepackage[a4paper,margin=3cm]{geometry}
\usepackage{booktabs}
\usepackage{array}
\usepackage{tikz}
\usetikzlibrary{calc}
\usepackage[colorlinks=true,linkcolor=blue,citecolor=blue,urlcolor=blue]{hyperref}

\DeclareMathOperator{\per}{per}
\DeclareMathOperator{\perfmat}{perfmat}
\DeclareMathOperator{\maxcut}{maxcut}
\DeclareMathOperator{\bip}{bip}
\DeclareMathOperator{\oct}{oct}

\newtheorem{theorem}{Theorem}[section]
\newtheorem{lemma}[theorem]{Lemma}
\newtheorem{corollary}[theorem]{Corollary}
\newtheorem{proposition}[theorem]{Proposition}
\newtheorem{problem}[theorem]{Problem}

\theoremstyle{definition}
\newtheorem{definition}[theorem]{Definition}
\newtheorem{example}[theorem]{Example}

\theoremstyle{remark}
\newtheorem{remark}[theorem]{Remark}

\newcommand{\PM}{\mathcal{PM}}
\newcommand{\Ccal}{\mathcal{C}}
\newcommand{\Jcal}{\mathcal{J}}

\begin{document}

\title[Odd-cycle defects in the Alon--Friedland bound]{Odd-Cycle Defects in the Alon--Friedland Bound}

\author{Mohsen Aliabadi}
\address{Department of Mathematics, Clayton State University, Morrow, GA, USA}
\email{mohsenaliabadi@gmail.com}
\email{mohsenaliabadi@clayton.edu}

\author{Elliot Krop}
\address{Department of Mathematics, Clayton State University, Morrow, GA, USA}
\email{elliotkrop@clayton.edu}


\begin{abstract}
Let \(G\) be a finite simple graph.  We study perfect matchings through two complementary viewpoints: reductions to bipartite permanent terms and the directed cycle covers counted by the ordinary adjacency permanent.  The main identity is an explicit odd-cycle-indexed form of the classical cycle-cover expansion: it separates \(\perfmat(G)^2\), the even-cycle-cover contribution coming from superposing two perfect matchings, from the contribution of cycle covers containing odd cycles.  This gives a nonnegative odd-cycle defect \(\delta_{\mathrm{odd}}(G)\).  Combining the identity with the Bregman--Minc inequality yields a structural refinement of the Alon--Friedland degree-sequence bound.  We prove product, positivity, and fractional-perfect-matching interpretations for the defect; show that for \(K_{2n}\) the defect asymptotically accounts for almost the entire Bregman--Minc target; and derive from this a derangement identity.  We also study near equality in the Alon--Friedland bound: we classify all one-edge perturbations of the extremal graphs, record the uniform obstruction coming from \(K_4\), and formulate a sharp bounded-degree gap problem whose two natural candidate extremal graphs cross between maximum degrees \(9\) and \(10\).
\end{abstract}

\maketitle

\section{Introduction and notation}

All graphs in this note are finite and simple.  For background on perfect matchings we refer to Lov\'asz and Plummer \cite{LovaszPlummer}.  If \(G\) is a graph, then \(V(G)\) and \(E(G)\) denote its vertex and edge sets, \(A(G)\) denotes its ordinary adjacency matrix, \(d_G(v)\) denotes the degree of a vertex \(v\), \(N_G(v)\) denotes the neighborhood of \(v\), and \(\Delta(G)\) denotes the maximum degree of \(G\).

A \emph{perfect matching} in \(G\) is a set of pairwise disjoint edges covering every vertex.  We write
\[
        \PM(G)=\{\text{perfect matchings of }G\},
        \qquad
        \perfmat(G)=|\PM(G)|.
\]
The empty graph is understood to have exactly one perfect matching.

If \(B=(b_{ij})\) is an \(n\times n\) matrix, its permanent is
\[
        \per(B)=\sum_{\sigma\in S_n}\prod_{i=1}^n b_{i,\sigma(i)},
\]
where \(S_n\) is the symmetric group on \(\{1,\dots,n\}\).  If \(G\) is a bipartite graph with bipartition \(X\sqcup Y\), we call \(G\) \emph{balanced} if \(|X|=|Y|\), and \emph{unbalanced} otherwise.  If \(G\) is balanced bipartite and \(B\) is its bipartite adjacency matrix, then
\[
        \perfmat(G)=\per(B).
\]
On the other hand, \(\per(A(G))\) counts the directed cycle covers of \(G\).  The square \(\perfmat(G)^2\), rather than \(\perfmat(G)\) itself, is the natural matching quantity in this setting because an ordered pair of perfect matchings superposes to a spanning union of doubled edges and even alternating cycles.

The Alon--Friedland bound \cite{AF} states that if \(G\) has no isolated vertices, then
\begin{equation}\label{eq:AF}
        \perfmat(G)\le U(G),
        \qquad
        U(G)=\prod_{v\in V(G)}(d_G(v)!)^{1/(2d_G(v))}.
\end{equation}
Moreover equality holds if and only if \(G\) is a disjoint union of balanced complete bipartite graphs \cite{AF}.  The estimate follows from the Bregman--Minc inequality \cite{Bregman,Minc}; see also the entropy proof of Radhakrishnan \cite{Radhakrishnan}.  Namely, if \(B\) is a square \(0\)-\(1\) matrix with row sums \(r_1,\dots,r_n\), then
\begin{equation}\label{eq:bregman}
        \per(B)\le \prod_{i=1}^n (r_i!)^{1/r_i},
\end{equation}
with the convention that the right-hand side is zero if some \(r_i=0\).

This note isolates two exact decompositions behind \eqref{eq:AF}.  The first decomposes \(\perfmat(G)\) after deleting edges or vertices to make the graph bipartite.  The second decomposes \(\per(A(G))\) according to the odd cycles appearing in a directed cycle cover.  These decompositions lead to an explicit nonnegative odd-cycle defect term in the Alon--Friedland bound.

The second decomposition should be read in relation to the classical proof of Alon and Friedland.  The permanent \(\per(A(G))\) counts directed cycle covers, and \(\perfmat(G)^2\) counts ordered pairs of perfect matchings, equivalently the even-cycle-cover part obtained by superposing two perfect matchings; this is precisely the mechanism used in \cite{AF} and in Galvin's exposition of the entropy-counting method \cite[Sec.~5.3]{Galvin}.  Thus we do not claim novelty for the underlying cycle-cover interpretation.  The point of Theorem~\ref{thm:odd-expansion} is to package the remaining terms explicitly by the odd cycles they contain, thereby producing a defect term that can be studied structurally and, in some families, asymptotically.  We have not found a prior use of this exact odd-cycle-indexed form for the Alon--Friedland problem.

The refinement should also be understood in this calibrated sense.  Since \(\delta_{\mathrm{odd}}(G)\) is itself expressed through perfect-matching counts of induced subgraphs, it is not presented as a general-purpose efficient way to compute \(\perfmat(G)\).  Its value is structural: it identifies the part of the adjacency permanent that cannot come from two perfect matchings, gives exact multiplicative information when enough defect terms can be evaluated, and explains why the Bregman--Minc target is very far from \(\perfmat(G)^2\) for dense nonbipartite graphs such as \(K_{2n}\).  The finite examples below are illustrative; the complete-graph calculation gives the main asymptotic use case.

The main consequences in this note are the structural criterion for when the defect is positive, the asymptotic saturation result for complete graphs, the classification of one-edge perturbations of the Alon--Friedland extremal graphs, and a sharp bounded-degree gap problem.  In particular, \(K_4\) gives a uniform obstruction \(c(\Delta)\le 1-(3/4)^{1/3}\) for every \(\Delta\ge 3\), while one-edge perturbations give the asymptotic obstruction of order \(1/\Delta\).  The two natural candidate extremal ratios cross between \(\Delta=9\) and \(\Delta=10\); this numerical crossover motivates, but does not prove, the sharp-gap problem formulated in Section~\ref{sec:defect}.

For \(F\subseteq E(G)\), let \(G-F\) denote the spanning subgraph obtained by deleting the edges in \(F\).  For \(S\subseteq V(G)\), let \(G-S\) denote the induced subgraph obtained by deleting the vertices in \(S\).  If \(J\subseteq E(G)\), let \(V(J)\) be the set of endpoints of the edges in \(J\).  A matching \(J\) \emph{saturates} a vertex \(v\) if \(v\in V(J)\).

We define
\[
        \bip(G)=\min\{|F|:G-F\text{ is bipartite}\},
\]
the \emph{edge-bipartization number} of \(G\).  Equivalently,
\[
        \bip(G)=|E(G)|-\maxcut(G),
\]
where \(\maxcut(G)\) is the maximum number of edges crossing a cut in \(G\), because every spanning bipartite subgraph is supported on a cut.  We also define
\[
        \oct(G)=\min\{|S|:G-S\text{ is bipartite}\},
\]
the \emph{odd-cycle transversal number} of \(G\).

The paper has four sections.  Section~\ref{sec:bip} contains auxiliary edge- and vertex-bipartization decompositions.  They are not needed for the proof of the odd-cycle expansion, but they explain how nonbipartite perfect-matching counts, including the subgraph counts appearing in the defect, may be reduced to sums of bipartite permanent terms when a small edge- or vertex-bipartization set is known.  Section~\ref{sec:odd} proves the exact odd-cycle expansion of \(\per(A(G))\) and records the structural consequences of the resulting defect identity.  Section~\ref{sec:defect} gives the defect-refined Alon--Friedland estimate, formulates the uniform-gap question, proves the near-extremal obstruction, and then gives examples, the complete-graph derangement identity, and the asymptotic consequences.

\section{Bipartization decompositions}\label{sec:bip}

This section is auxiliary to the odd-cycle expansion.  Its purpose is to record elementary decompositions that turn perfect-matching counts in nonbipartite graphs into sums of bipartite permanent terms after deleting a small edge or vertex set.  These formulas help place the subgraph counts \(\perfmat(G-V(\Ccal))\) appearing later in a computationally familiar permanent framework.  The first identity is a partition of the set of perfect matchings according to which deleted edges are used.

\begin{proposition}\label{prop:edge-bip}
Let \(G\) be a finite simple graph, let \(F\subseteq E(G)\), and put \(H=G-F\).  Then
\begin{equation}\label{eq:edge-bip}
        \perfmat(G)=
        \sum_{\substack{J\subseteq F\\ J\text{ a matching}}}
        \perfmat(H-V(J)).
\end{equation}
\end{proposition}

\begin{proof}
Let \(M\in\PM(G)\).  Set
\[
        J=M\cap F,
        \qquad
        N=M\setminus J.
\]
Then \(J\) is a matching contained in \(F\), and \(N\) is a perfect matching of \(H-V(J)\).  Conversely, if \(J\subseteq F\) is a matching and \(N\in\PM(H-V(J))\), then \(J\cup N\) is a perfect matching of \(G\).  These constructions are inverse to one another.  Counting the resulting pairs \((J,N)\) proves \eqref{eq:edge-bip}.
\end{proof}

\begin{corollary}\label{cor:edge-terms}
Let \(F\subseteq E(G)\) be such that \(H=G-F\) is bipartite.  Then \(\perfmat(G)\) is a sum of the numbers \(\perfmat(H-V(J))\), where \(J\) runs over matchings contained in \(F\).  Each nonzero summand is the permanent of a bipartite adjacency matrix.  In particular, if \(|F|=\bip(G)\), then \(\perfmat(G)\) is expressed as a sum of at most \(2^{\bip(G)}\) bipartite permanent terms.
\end{corollary}

\begin{proof}
The identity \eqref{eq:edge-bip} gives
\[
        \perfmat(G)=
        \sum_{\substack{J\subseteq F\\ J\text{ a matching}}}
        \perfmat(H-V(J)).
\]
Since \(H\) is bipartite and induced subgraphs of bipartite graphs are bipartite, each graph \(H-V(J)\) is bipartite.  If \(H-V(J)\) is balanced, then its perfect matchings are counted by the permanent of its bipartite adjacency matrix; if it is unbalanced, then it has no perfect matching and the corresponding summand is zero.  Since the number of subsets of \(F\) is \(2^{|F|}\), the final assertion follows.
\end{proof}

For a graph \(Q\), define \(U(Q)\) by the product in \eqref{eq:AF} if \(Q\) has no isolated vertices, set \(U(\varnothing)=1\), and set \(U(Q)=0\) if \(Q\neq\varnothing\) has an isolated vertex.

\begin{corollary}\label{cor:edge-bound}
If \(F\subseteq E(G)\) and \(H=G-F\) is bipartite, then
\[
        \perfmat(G)
        \le
        \sum_{\substack{J\subseteq F\\ J\text{ a matching}}}
        U(H-V(J)).
\]
Consequently,
\[
        \perfmat(G)
        \le
        \min_{\substack{F\subseteq E(G)\\ G-F\text{ bipartite}}}
        \sum_{\substack{J\subseteq F\\ J\text{ a matching}}}
        U((G-F)-V(J)).
\]
\end{corollary}

\begin{proof}
By Proposition~\ref{prop:edge-bip}, it is enough to estimate each \(\perfmat(H-V(J))\).  If \(H-V(J)\) is unbalanced, then \(\perfmat(H-V(J))=0\le U(H-V(J))\).  If \(H-V(J)\neq\varnothing\) has an isolated vertex, then again \(\perfmat(H-V(J))=0=U(H-V(J))\).

It remains to consider the case in which \(H-V(J)\) is a balanced bipartite graph without isolated vertices.  Let its bipartition be \(X\sqcup Y\), and let \(B\) be the corresponding square bipartite adjacency matrix.  Applying \eqref{eq:bregman} to \(B\) and to \(B^{\mathsf T}\) gives
\[
        \per(B)^2
        =\per(B)\per(B^{\mathsf T})
        \le
        \prod_{x\in X}(d(x)!)^{1/d(x)}
        \prod_{y\in Y}(d(y)!)^{1/d(y)}
        =U(H-V(J))^2.
\]
Since \(\perfmat(H-V(J))=\per(B)\), it follows that
\[
        \perfmat(H-V(J))\le U(H-V(J)).
\]
Summing over \(J\) yields the first inequality, and the second follows by minimizing over all edge-bipartization sets.
\end{proof}

The vertex version uses an odd-cycle transversal.

\begin{definition}
Let \(S\subseteq V(G)\) be an odd-cycle transversal.  We write \(\Jcal_S\) for the set of all matchings \(J\) in \(G\) such that \(S\subseteq V(J)\) and every edge of \(J\) is incident with \(S\).  Equivalently, \(J\in\Jcal_S\) if and only if \(J\) saturates every vertex of \(S\) and no edge of \(J\) lies entirely in \(V(G)\setminus S\).
\end{definition}

\begin{proposition}[Vertex-bipartization identity]\label{prop:vertex-bip}
Let \(G\) be a finite simple graph, and let \(S\subseteq V(G)\) be an odd-cycle transversal.  Then
\begin{equation}\label{eq:vertex-bip}
        \perfmat(G)=
        \sum_{J\in\Jcal_S}\perfmat(G-V(J)).
\end{equation}
Moreover, for every \(J\in\Jcal_S\), the graph \(G-V(J)\) is bipartite.
\end{proposition}

\begin{proof}
Let \(M\in\PM(G)\), and define
\[
        J=\{e\in M:e\cap S\neq\varnothing\},
        \qquad
        N=M\setminus J.
\]
Since \(M\) is a perfect matching, every vertex of \(S\) is incident with exactly one edge of \(J\).  Hence \(S\subseteq V(J)\).  Every edge of \(J\) has at least one endpoint in \(S\), so \(J\in\Jcal_S\).  The remaining matching \(N\) covers exactly the vertices outside \(V(J)\), hence \(N\in\PM(G-V(J))\).

Conversely, let \(J\in\Jcal_S\) and \(N\in\PM(G-V(J))\).  Since \(J\) is a matching and \(N\) is a matching on the complementary vertex set \(V(G)\setminus V(J)\), the union \(J\cup N\) is a matching.  Because \(S\subseteq V(J)\) and \(N\) covers all vertices of \(G-V(J)\), the union \(J\cup N\) covers every vertex of \(G\).  Thus \(J\cup N\in\PM(G)\).

These constructions are inverse to each other, and counting the resulting pairs \((J,N)\) proves \eqref{eq:vertex-bip}.  Finally, since \(S\subseteq V(J)\), the graph \(G-V(J)\) is an induced subgraph of \(G-S\).  As \(G-S\) is bipartite, so is \(G-V(J)\).
\end{proof}

\begin{corollary}\label{cor:vertex-count}
For every odd-cycle transversal \(S\),
\[
        |\Jcal_S|\le \prod_{s\in S} d_G(s)\le \Delta(G)^{|S|}.
\]
Consequently, if \(|S|=\oct(G)\), then \(\perfmat(G)\) is expressed as a sum of at most \(\Delta(G)^{\oct(G)}\) bipartite permanent terms.
\end{corollary}

\begin{proof}
Each matching \(J\in\Jcal_S\) assigns to each \(s\in S\) its unique partner in \(J\).  Hence the map
\[
        J\longmapsto (\text{partner of }s\text{ in }J)_{s\in S}
\]
is injective into \(\prod_{s\in S} N_G(s)\).  Indeed, every edge of \(J\) is incident with \(S\), so the list of partners reconstructs the whole edge set \(J\), including the case in which an edge of \(J\) has both endpoints in \(S\).  Therefore
\[
        |\Jcal_S|\le \prod_{s\in S} |N_G(s)|=\prod_{s\in S} d_G(s)\le \Delta(G)^{|S|}.
\]
The final assertion follows from Proposition~\ref{prop:vertex-bip}.
\end{proof}

\begin{example}[Internal perturbations of \(K_{n,n}\)]\label{ex:internal}
Let \(X\) and \(Y\) be disjoint sets with \(|X|=|Y|=n\), and let \(K_{X,Y}\) denote the complete bipartite graph with bipartition \(X\sqcup Y\).  Let \(L_X\) be a graph on \(X\), let \(L_Y\) be a graph on \(Y\), and set
\[
        G=K_{X,Y}\cup L_X\cup L_Y.
\]
If \(m_j(L)\) denotes the number of \(j\)-edge matchings in \(L\), then
\[
        \perfmat(G)=
        \sum_{j=0}^{\lfloor n/2\rfloor}
        m_j(L_X)m_j(L_Y)(n-2j)!.
\]
Indeed, apply Proposition~\ref{prop:edge-bip} with \(F=E(L_X)\cup E(L_Y)\) and \(H=K_{X,Y}\).  A matching \(J\subseteq F\) decomposes uniquely as \(J_X\sqcup J_Y\).  The graph \(H-V(J)\) has a perfect matching exactly when \(|J_X|=|J_Y|=j\), in which case it is \(K_{n-2j,n-2j}\), which has \((n-2j)!\) perfect matchings.
\end{example}

\section{The odd-cycle expansion of the adjacency permanent}\label{sec:odd}

Fix a labelling \(V(G)=\{1,\dots,N\}\).  A \emph{directed cycle cover} of \(G\) is a collection of vertex-disjoint directed cycles covering all vertices, where each directed edge is an orientation of an edge of \(G\).  Equivalently, it is a permutation \(\sigma\in S_N\) such that \(i\sigma(i)\in E(G)\) for every \(i\).  Hence \(\per(A(G))\) counts the directed cycle covers of \(G\), the standard permanent interpretation used in the proof of \eqref{eq:AF}.

If \(D\) is a directed even cycle
\[
        v_0\to v_1\to \cdots \to v_{2m-1}\to v_0,
\]
we write the cyclic order so that \(v_0\) is the smallest labelled vertex on the cycle.  Its \emph{canonical split} is the ordered pair of matchings
\[
        M_1(D)=\{v_0v_1,v_2v_3,\dots,v_{2m-2}v_{2m-1}\},
\]
\[
        M_2(D)=\{v_1v_2,v_3v_4,\dots,v_{2m-1}v_0\}.
\]
A directed \(2\)-cycle is treated as a doubled edge; in that case both split matchings contain the same edge.

\begin{definition}
An \emph{odd-cycle family} in \(G\) is a collection \(\Ccal\) of pairwise vertex-disjoint odd cycles in \(G\).  The empty family is allowed, and we write
\[
        V(\Ccal)=\bigcup_{C\in\Ccal}V(C).
\]
\end{definition}

\begin{lemma}\label{lem:split}
Let \(H\) be a labelled finite simple graph.  Ordered pairs \((M_1,M_2)\in \PM(H)\times \PM(H)\) are in bijection with directed cycle covers of \(H\) all of whose directed cycles have even length.
\end{lemma}

\begin{proof}
Given \((M_1,M_2)\), form the \(2\)-regular multigraph \(M_1\cup M_2\), counting an edge twice if it lies in both matchings.  Its connected components are even cycles, with a common edge regarded as a directed \(2\)-cycle.  On each component of length at least \(4\), the edges alternate between \(M_1\) and \(M_2\).  Choose the unique orientation for which, at the smallest labelled vertex \(v_0\) of the component, the first directed edge belongs to \(M_1\).  This produces a directed cycle cover with only even cycles.

Conversely, given a directed cycle cover all of whose cycles are even, split each directed even cycle by the canonical rule above.  The first split edges over all cycles form a perfect matching \(M_1\), and the second split edges form a perfect matching \(M_2\).  The two constructions are inverse to one another.
\end{proof}

The following identity is a bookkeeping form of the standard cycle-cover expansion of the permanent.  Its empty-family term is the usual two-matching superposition contribution; the remaining terms are grouped according to the odd cycles in the directed cycle cover.  The later defect terminology refers to this grouping, not to a new interpretation of the permanent itself.

\begin{theorem}[Exact odd-cycle expansion]\label{thm:odd-expansion}
Let \(G\) be a finite simple graph, and let \(A=A(G)\) be its adjacency matrix.  Then
\begin{equation}\label{eq:odd-expansion}
        \per(A)=
        \sum_{\Ccal}
        2^{|\Ccal|}\perfmat(G-V(\Ccal))^2,
\end{equation}
where the sum ranges over all odd-cycle families \(\Ccal\) in \(G\), including the empty family.
\end{theorem}

\begin{proof}
We construct a bijection between directed cycle covers of \(G\) and the objects counted by the right-hand side of \eqref{eq:odd-expansion}.

Start with an odd-cycle family \(\Ccal\), choose one of the two cyclic orientations of each odd cycle in \(\Ccal\), and choose an ordered pair
\[
        (M_1,M_2)\in \PM(G-V(\Ccal))\times \PM(G-V(\Ccal)).
\]
By Lemma~\ref{lem:split}, the pair \((M_1,M_2)\) determines a directed cycle cover of \(G-V(\Ccal)\) consisting only of even cycles.  Together with the chosen orientations of the odd cycles in \(\Ccal\), this gives a directed cycle cover of \(G\).

Conversely, let \(\mathcal{D}\) be a directed cycle cover of \(G\).  Let \(\Ccal\) be the family of underlying undirected odd cycles occurring in \(\mathcal{D}\).  Their orientations are already part of \(\mathcal{D}\).  After deleting \(V(\Ccal)\), the remaining directed cycles all have even length.  By Lemma~\ref{lem:split}, this remaining even directed cycle cover determines a unique ordered pair
\[
        (M_1,M_2)\in \PM(G-V(\Ccal))\times \PM(G-V(\Ccal)).
\]
The two constructions are inverse.  For a fixed \(\Ccal\), there are \(2^{|\Ccal|}\) choices of orientations of its odd cycles and \(\perfmat(G-V(\Ccal))^2\) ordered pairs of perfect matchings on the remaining graph.  Summing over \(\Ccal\) proves \eqref{eq:odd-expansion}.
\end{proof}

\begin{definition}\label{def:defect}
The \emph{odd-cycle defect} of \(G\) is
\[
        \delta_{\mathrm{odd}}(G)=
        \sum_{\substack{\Ccal\neq\varnothing\\
        \Ccal\text{ an odd-cycle family}}}
        2^{|\Ccal|}\perfmat(G-V(\Ccal))^2.
\]
Thus Theorem~\ref{thm:odd-expansion} yields the exact identity
\begin{equation}\label{eq:per-defect}
        \per(A(G))=\perfmat(G)^2+\delta_{\mathrm{odd}}(G).
\end{equation}
\end{definition}

\begin{remark}[Relation with the standard even-cycle expansion]\label{rem:standard-even-expansion}
The case with no odd cycles is the familiar superposition principle for perfect matchings: an ordered pair of perfect matchings decomposes into doubled edges and even alternating cycles, and each even cycle can be split in two ways.  This is the even-cycle part of the cycle-cover expansion used in the Alon--Friedland argument \cite{AF}; see also \cite[Sec.~5.3]{Galvin}.  Theorem~\ref{thm:odd-expansion} adds no hidden weights beyond this classical cycle-cover count; its role is to separate the odd-cycle terms explicitly.
\end{remark}

\begin{remark}\label{rem:parity}
If \(|V(G)|\) is even, then a term indexed by \(\Ccal\) can be nonzero only if \(|\Ccal|\) is even.  Indeed, since every cycle in \(\Ccal\) has odd length, the parity of \(|V(\Ccal)|\) equals the parity of \(|\Ccal|\).  If \(|\Ccal|\) is odd and \(|V(G)|\) is even, then \(G-V(\Ccal)\) has odd order and therefore no perfect matching.  This parity observation will be used implicitly in the complete-graph calculation in Proposition~\ref{prop:derangement}.
\end{remark}

\begin{remark}\label{rem:bipartite-case}
If \(G\) is bipartite, then it has no nonempty odd-cycle family, so \(\delta_{\mathrm{odd}}(G)=0\).  In this case \eqref{eq:per-defect} reduces to the familiar identity
\[
        \per(A(G))=\perfmat(G)^2
\]
when the two bipartition classes are balanced.  For nonbipartite graphs, \(\delta_{\mathrm{odd}}(G)\) measures only the contribution of directed cycle covers containing odd cycles; it is not the same as edge-distance or vertex-distance from bipartiteness.
\end{remark}

\medskip\noindent\textbf{Structural consequences of the defect identity.}

It is useful to write
\[
        p(G)=\perfmat(G),\qquad
        \delta(G)=\delta_{\mathrm{odd}}(G),\qquad
        \Pi(G)=\per(A(G)).
\]
Then \eqref{eq:per-defect} says \(\Pi(G)=p(G)^2+\delta(G)\).

\begin{proposition}\label{prop:product-defect}
If \(G_1\) and \(G_2\) are vertex-disjoint finite simple graphs, then
\[
        \delta(G_1\sqcup G_2)
        =\delta(G_1)\delta(G_2)
        +p(G_1)^2\delta(G_2)
        +p(G_2)^2\delta(G_1).
\]
\end{proposition}

\begin{proof}
A perfect matching of \(G_1\sqcup G_2\) is exactly a pair of perfect matchings, one in \(G_1\) and one in \(G_2\).  Hence
\[
        p(G_1\sqcup G_2)=p(G_1)p(G_2).
\]
Also, \(A(G_1\sqcup G_2)\) is block diagonal with diagonal blocks \(A(G_1)\) and \(A(G_2)\).  Therefore every permanent-contributing permutation preserves the two blocks, and
\[
        \Pi(G_1\sqcup G_2)=\Pi(G_1)\Pi(G_2).
\]
Using \(\Pi(G)=p(G)^2+\delta(G)\), we obtain
\[
        p(G_1)^2p(G_2)^2+\delta(G_1\sqcup G_2)
        =(p(G_1)^2+\delta(G_1))(p(G_2)^2+\delta(G_2)).
\]
Subtracting \(p(G_1)^2p(G_2)^2\) gives the formula.
\end{proof}

\begin{corollary}\label{cor:ratio-mult}
If \(\Pi(G)>0\), define
\[
        r(G)=\frac{p(G)^2}{\Pi(G)}.
\]
Then \(r\) is multiplicative over connected components.  In particular, if \(G=\bigsqcup_t G_t\) and \(\Pi(G)>0\), then
\[
        r(G)=\prod_t r(G_t).
\]
\end{corollary}

\begin{proof}
Both \(p\) and \(\Pi\) are multiplicative over disjoint unions, as in the proof of Proposition~\ref{prop:product-defect}.
\end{proof}

\begin{proposition}[Structural criterion for positive defect]\label{prop:positive-defect-criterion}
For every finite simple graph \(G\), one has \(\delta_{\mathrm{odd}}(G)>0\) if and only if \(G\) has a spanning subgraph in which every component is either a single edge or an odd cycle, with at least one odd-cycle component.
\end{proposition}

\begin{proof}
By Definition~\ref{def:defect}, every summand in \(\delta_{\mathrm{odd}}(G)\) is nonnegative.  Hence \(\delta_{\mathrm{odd}}(G)>0\) if and only if there is a nonempty odd-cycle family \(\Ccal\) such that
\[
        \perfmat(G-V(\Ccal))\ge 1.
\]
Suppose such a family \(\Ccal\) exists, and let \(M\) be a perfect matching of \(G-V(\Ccal)\).  Then
\[
        \left(\bigcup_{C\in\Ccal} C\right)\cup M
\]
is a spanning subgraph of \(G\) whose components are precisely the odd cycles in \(\Ccal\) and the single-edge components of \(M\).  Since \(\Ccal\neq\varnothing\), at least one component is an odd cycle.

Conversely, suppose \(G\) has a spanning subgraph whose components are single edges and odd cycles, with at least one odd-cycle component.  Let \(\Ccal\) be the family of its odd-cycle components.  The single-edge components then form a perfect matching of \(G-V(\Ccal)\).  Thus \(\perfmat(G-V(\Ccal))\ge 1\), so the corresponding summand in \(\delta_{\mathrm{odd}}(G)\) is positive.
\end{proof}

\begin{example}\label{ex:criterion}
The criterion explains why \(K_4\) has zero odd-cycle defect although it is nonbipartite.  Every odd cycle in \(K_4\) is a triangle, and deleting a triangle leaves one vertex, which has no perfect matching.  Hence no spanning subgraph of \(K_4\) can have components consisting only of single edges and at least one odd cycle.  By contrast, \(C_5\) itself is such a spanning subgraph, so \(\delta_{\mathrm{odd}}(C_5)>0\).
\end{example}

\begin{remark}[Fractional perfect matchings]\label{rem:fractional-perfect-matchings}
Proposition~\ref{prop:positive-defect-criterion} also has a standard polyhedral interpretation.  Let
\[
        P_f(G)=\{x\in \mathbb{R}_{\ge 0}^{E(G)}: x(\delta(v))=1 \text{ for every } v\in V(G)\}
\]
be the fractional perfect matching polytope.  Its extreme points are half-integral: their supports are disjoint unions of single edges, assigned value \(1\), and odd cycles, assigned value \(1/2\); see, for example, \cite[Ch.~7]{LovaszPlummer}.  Consequently \(\delta_{\mathrm{odd}}(G)>0\) if and only if \(P_f(G)\) has a nonintegral extreme point.  Thus the defect detects exactly the basic half-integral perfect-matchings that contain at least one odd cycle.
\end{remark}

\section{The odd-cycle defect and the Alon--Friedland bound}\label{sec:defect}

The odd-cycle expansion gives a direct refinement of the usual Alon--Friedland estimate whenever the odd-cycle defect is positive.  The refinement is mainly structural and asymptotic.  In general, evaluating \(\delta_{\mathrm{odd}}(G)\) may be as difficult as evaluating perfect-matching counts in many induced subgraphs.  Nevertheless, when a useful part of the defect is visible, or when the full defect can be evaluated in a symmetric family, the identity gives information about where the Alon--Friedland slack comes from.

\begin{theorem}[Odd-cycle defect bound]\label{thm:defect-bound}
Let \(G\) be a finite simple graph with no isolated vertices.  Then
\begin{equation}\label{eq:defect-bound}
        \perfmat(G)^2+\delta_{\mathrm{odd}}(G)\le U(G)^2.
\end{equation}
\end{theorem}

\begin{proof}
By \eqref{eq:per-defect}, the left-hand side equals \(\per(A(G))\).  The row sum of the row of \(A(G)\) indexed by \(v\) is \(d_G(v)\).  Applying \eqref{eq:bregman} to \(A(G)\) yields
\[
        \per(A(G))
        \le
        \prod_{v\in V(G)} (d_G(v)!)^{1/d_G(v)}
        =
        U(G)^2.
\]
This proves \eqref{eq:defect-bound}.
\end{proof}

\begin{corollary}\label{cor:AF-defect}
Let \(G\) be a finite simple graph of even order and with no isolated vertices.  Then
\[
        \perfmat(G)\le \bigl(U(G)^2-\delta_{\mathrm{odd}}(G)\bigr)^{1/2}.
\]
In particular, if \(\delta_{\mathrm{odd}}(G)>0\), then \(\perfmat(G)<U(G)\).
\end{corollary}

\begin{proof}
This follows immediately by rearranging \eqref{eq:defect-bound}.
\end{proof}

\begin{corollary}[Multiplicative defect form]\label{cor:multiplicative-defect-form}
Let \(G\) be a finite simple graph of even order and with no isolated vertices.  If \(\per(A(G))>0\), then
\[
        \perfmat(G)
        \le
        U(G)
        \left(
        1-\frac{\delta_{\mathrm{odd}}(G)}{\per(A(G))}
        \right)^{1/2}.
\]
\end{corollary}

\begin{proof}
By \eqref{eq:per-defect},
\[
        \perfmat(G)^2
        =
        \per(A(G))
        \left(
        1-\frac{\delta_{\mathrm{odd}}(G)}{\per(A(G))}
        \right).
\]
The Bregman--Minc estimate applied to \(A(G)\) gives \(\per(A(G))\le U(G)^2\).  Taking square roots proves the claim.
\end{proof}

\medskip\noindent\textbf{The uniform-gap question.}

The equality cases in \eqref{eq:AF} are exactly the disjoint unions of balanced complete bipartite graphs.  It is therefore natural to ask whether, after fixing the maximum degree, every non-extremal graph is separated from equality by a uniform multiplicative gap.

\begin{problem}\label{prob:uniform-gap}
Fix \(\Delta\ge 3\).  Does there exist \(c(\Delta)>0\) such that every finite simple graph \(G\) of even order, with no isolated vertices and maximum degree at most \(\Delta\), which is not a disjoint union of balanced complete bipartite graphs, satisfies
\[
        \perfmat(G)\le (1-c(\Delta))U(G)?
\]
Equivalently, how close to \(1\) can \(\perfmat(G)/U(G)\) be for connected non-extremal graphs of maximum degree at most \(\Delta\)?
\end{problem}

\begin{proposition}[Near-extremal connected non-extremal graphs]\label{prop:near-extremal}
For \(r\ge 2\), let \(G_r\) be the graph obtained from \(K_{r,r}\) by adding one edge joining two vertices in the same part.  Then \(G_r\) is connected, has even order \(2r\), has no isolated vertices, and is not a disjoint union of balanced complete bipartite graphs.  Moreover
\[
        \frac{\perfmat(G_r)}{U(G_r)}
        =\frac{(r!)^{1/(r(r+1))}}{(r+1)^{1/(r+1)}}
        =1-\frac{1}{r+1}+O\left(\frac{\log r}{r^2}\right).
\]
In particular,
\[
        \frac{\perfmat(G_r)}{U(G_r)}\longrightarrow 1.
\]
\end{proposition}

\begin{proof}
Let \(X\sqcup Y\) be the bipartition of \(K_{r,r}\), and suppose the added edge is \(e=x_1x_2\) with \(x_1,x_2\in X\).  A perfect matching of \(G_r\) either omits \(e\), in which case it is a perfect matching of \(K_{r,r}\), or uses \(e\).  If it uses \(e\), then the remaining \(r-2\) vertices of \(X\) would have to be matched to all \(r\) vertices of \(Y\), which is impossible.  Hence
\[
        \perfmat(G_r)=r!.
\]

The added edge raises the degrees of \(x_1\) and \(x_2\) from \(r\) to \(r+1\), while the other \(2r-2\) vertices still have degree \(r\).  Therefore
\[
        U(G_r)=((r+1)!)^{1/(r+1)}(r!)^{(r-1)/r}.
\]
It follows that
\[
        \frac{\perfmat(G_r)}{U(G_r)}
        =\frac{r!}{((r+1)!)^{1/(r+1)}(r!)^{(r-1)/r}}
        =\frac{(r!)^{1/(r(r+1))}}{(r+1)^{1/(r+1)}}.
\]
Taking logarithms and applying Stirling's formula,
\[
        \log\frac{\perfmat(G_r)}{U(G_r)}
        =\frac{\log(r!)}{r(r+1)}-\frac{\log(r+1)}{r+1}
        =-\frac{1}{r+1}+O\left(\frac{\log r}{r^2}\right).
\]
Exponentiating gives
\[
        \frac{\perfmat(G_r)}{U(G_r)}
        =1-\frac{1}{r+1}+O\left(\frac{\log r}{r^2}\right).
\]
Finally, \(G_r\) contains triangles of the form \(x_1yx_2x_1\), with \(y\in Y\), so it is not bipartite and hence is not a disjoint union of balanced complete bipartite graphs.
\end{proof}

\begin{lemma}[Monotonicity of the one-edge ratios]\label{lem:rho-monotone}
For \(r\ge 1\), put
\[
        \rho_r=
        \frac{(r!)^{1/(r(r+1))}}{(r+1)^{1/(r+1)}}.
\]
Then the sequence \((\rho_r)_{r\ge 1}\) is strictly increasing.
\end{lemma}

\begin{proof}
Put \(a_r=-\log \rho_r\).  It is enough to prove \(a_r>a_{r+1}\).  A direct calculation gives
\[
a_r-a_{r+1}
=
\frac{
r(r+3)\log(r+1)-r(r+1)\log(r+2)-2\log(r!)
}{
r(r+1)(r+2)
}.
\]
By the arithmetic-geometric mean inequality,
\[
        r!\le \left(\frac{r+1}{2}\right)^r .
\]
Hence the numerator above is at least
\[
        r\left(2\log 2-(r+1)\log\left(1+\frac{1}{r+1}\right)\right).
\]
Since \((r+1)\log(1+1/(r+1))<1<2\log 2\), this lower bound is positive.  Therefore \(a_r>a_{r+1}\), and so \(\rho_r<\rho_{r+1}\).
\end{proof}

\begin{proposition}[One-edge perturbations of extremal graphs]\label{prop:one-edge-perturbations}
Let
\[
        Q=\bigsqcup_{i=1}^m K_{r_i,r_i}
\]
be a disjoint union of balanced complete bipartite graphs, and let \(G=Q+e\) be obtained by adding one new edge.  Then \(\perfmat(G)=\perfmat(Q)\).  More precisely:
\begin{enumerate}
    \item if \(e\) joins two vertices in the same part of a component \(K_{r_i,r_i}\), then
    \[
            \frac{\perfmat(G)}{U(G)}=\rho_{r_i};
    \]
    \item if \(e\) joins a vertex of \(K_{r_i,r_i}\) to a vertex of a distinct component \(K_{r_j,r_j}\), then
    \[
            \frac{\perfmat(G)}{U(G)}
            =
            \sqrt{\rho_{r_i}\rho_{r_j}}.
    \]
\end{enumerate}
Consequently, if \(\Delta(G)\le \Delta\), then
\[
        \frac{\perfmat(G)}{U(G)}\le \rho_{\Delta-1}.
\]
\end{proposition}

\begin{proof}
Since \(e\) is a new edge, its endpoints cannot lie in opposite parts of the same complete bipartite component.

First suppose that \(e\) joins two vertices in the same part of a component \(K_{r_i,r_i}\).  If a perfect matching used \(e\), then after deleting its endpoints the two sides of that component would have sizes \(r_i-2\) and \(r_i\), so the remaining graph could not have a perfect matching.  Thus every perfect matching of \(G\) avoids \(e\), and \(\perfmat(G)=\perfmat(Q)\).  The only change in \(U\) is that the two endpoints have degree \(r_i+1\) instead of \(r_i\).  Therefore
\[
        \frac{\perfmat(G)}{U(G)}
        =
        \frac{r_i!}{((r_i+1)!)^{1/(r_i+1)}(r_i!)^{(r_i-1)/r_i}}
        =
        \rho_{r_i}.
\]

Now suppose that \(e\) joins a vertex of \(K_{r_i,r_i}\) to a vertex of a distinct component \(K_{r_j,r_j}\).  If a perfect matching used \(e\), deleting its endpoints would leave two odd-order components, so again no perfect matching can use \(e\).  Thus \(\perfmat(G)=\perfmat(Q)\).  The two changed degree factors give
\[
        \frac{\perfmat(G)}{U(G)}
        =
        \left(
        \frac{(r_i!)^{1/(2r_i)}}{((r_i+1)!)^{1/(2(r_i+1))}}
        \right)
        \left(
        \frac{(r_j!)^{1/(2r_j)}}{((r_j+1)!)^{1/(2(r_j+1))}}
        \right)
        =
        \sqrt{\rho_{r_i}\rho_{r_j}}.
\]
The final assertion follows from Lemma~\ref{lem:rho-monotone}, because every endpoint whose degree is raised from \(r\) to \(r+1\) satisfies \(r+1\le \Delta\).
\end{proof}

\begin{corollary}\label{cor:no-absolute-gap}
There is no absolute constant \(c>0\) such that every finite simple graph \(G\) of even order, with no isolated vertices, which is not a disjoint union of balanced complete bipartite graphs, satisfies
\[
        \perfmat(G)\le (1-c)U(G).
\]
Moreover, any constant \(c(\Delta)\) valid in Problem~\ref{prob:uniform-gap} must satisfy
\[
        c(\Delta)\le 1-\left(\frac34\right)^{1/3}
\]
for every \(\Delta\ge 3\), and also
\[
        c(\Delta)\le \frac{1}{\Delta}+O\left(\frac{\log\Delta}{\Delta^2}\right)
\]
as \(\Delta\to\infty\).
\end{corollary}

\begin{proof}
The first assertion follows from Proposition~\ref{prop:near-extremal}, since the ratios \(\perfmat(G_r)/U(G_r)\) tend to \(1\).  The uniform numerical obstruction comes from \(K_4\): it has maximum degree \(3\), is not a disjoint union of balanced complete bipartite graphs, and
\[
        \frac{\perfmat(K_4)}{U(K_4)}
        =
        \frac{3}{6^{2/3}}
        =
        \left(\frac34\right)^{1/3}.
\]
Thus, for every \(\Delta\ge 3\), any valid constant must satisfy
\[
        1-c(\Delta)\ge \left(\frac34\right)^{1/3}.
\]
For the asymptotic assertion, put \(\Delta=r+1\) in Proposition~\ref{prop:near-extremal}.  If a bound with constant \(c(\Delta)\) holds for all graphs of maximum degree at most \(\Delta\), then
\[
        1-c(\Delta)\ge \frac{\perfmat(G_{\Delta-1})}{U(G_{\Delta-1})}
        =1-\frac{1}{\Delta}+O\left(\frac{\log\Delta}{\Delta^2}\right),
\]
and the stated asymptotic upper bound on \(c(\Delta)\) follows.
\end{proof}

\begin{problem}[Sharp uniform gap problem]\label{prob:sharp-uniform-gap}
For \(\Delta\ge 3\), define
\[
        R_\Delta
        =
        \sup
        \left\{
        \frac{\perfmat(G)}{U(G)}:
        \begin{array}{l}
        G \text{ has even order, no isolated vertices, and } \Delta(G)\le \Delta,\\
        G \text{ is not a disjoint union of balanced complete bipartite graphs}
        \end{array}
        \right\}.
\]
Is it true that
\[
        R_\Delta
        =
        \max\left\{
        \left(\frac34\right)^{1/3},
        \rho_{\Delta-1}
        \right\},
        \qquad
        \rho_r=
        \frac{(r!)^{1/(r(r+1))}}{(r+1)^{1/(r+1)}}?
\]
Equivalently, is \(K_4\) extremal for \(3\le \Delta\le 9\), and is the graph \(G_{\Delta-1}\) from Proposition~\ref{prop:near-extremal} extremal for \(\Delta\ge 10\)?
\end{problem}

\begin{remark}[Candidate values and small-graph evidence]\label{rem:sharp-gap-evidence}
The simpler candidate formula with only \(\rho_{\Delta-1}\) cannot hold for small \(\Delta\), because \(K_4\) is an admissible non-extremal graph for every \(\Delta\ge 3\) and
\[
        \frac{\perfmat(K_4)}{U(K_4)}
        =
        \left(\frac34\right)^{1/3}
        \approx 0.908560.
\]
On the other hand,
\[
        \rho_8\approx 0.907692
        \quad\text{and}\quad
        \rho_9\approx 0.915746.
\]
Thus, if the formula in Problem~\ref{prob:sharp-uniform-gap} is correct, the change in the extremal example occurs between maximum degrees \(9\) and \(10\).  We view this as a candidate transition, not as computational evidence by itself.

As a small consistency check, we performed an exhaustive enumeration of all connected simple graphs on at most seven vertices, and then restricted to even order, no isolated vertices, at least one perfect matching, and non-extremal connected graphs.  Perfect matchings were counted by a direct bitmask recursion, and \(U(G)\) was computed from the degree sequence.  Among the resulting graphs, \(K_4\) is uniquely maximal.  The runner-up is \(K_{3,3}\) with one edge deleted.  For this graph \(p=4\), two vertices have degree \(2\), and four vertices have degree \(3\), so
\[
        \frac{\perfmat}{U}
        =
        \frac{4}{2^{1/2}6^{2/3}}
        \approx 0.856599.
\]
This search is deliberately modest: it supports the formulation of the problem but does not constitute serious evidence for all \(\Delta\) or all orders.
\end{remark}

\begin{proposition}[One-sided internal perturbations]\label{prop:one-sided-internal}
Let \(X\sqcup Y\) be the bipartition of \(K_{r,r}\), with \(r\ge 2\).  Let \(L_X\) be a nonempty simple graph on \(X\), and set
\[
        G=K_{r,r}\cup L_X.
\]
Then
\[
        \perfmat(G)=r!,
\]
and among such one-sided perturbations, the ratio \(\perfmat(G)/U(G)\) is maximized when \(L_X\) consists of a single edge.  In that case the ratio is \(\rho_r\).
\end{proposition}

\begin{proof}
No edge of \(L_X\) can belong to a perfect matching, because using such an edge would leave fewer vertices on the \(X\)-side than on the \(Y\)-side to be matched by cross edges.  Hence all perfect matchings are exactly the \(r!\) perfect matchings of \(K_{r,r}\).

It remains only to compare \(U(G)\).  The function \(d\mapsto (d!)^{1/(2d)}\) is strictly increasing for positive integers \(d\), since the geometric mean of \(1,\ldots,d+1\) is larger than the geometric mean of \(1,\ldots,d\).  Each added edge of \(L_X\) strictly increases \(U(G)\) while leaving \(\perfmat(G)\) unchanged.  Hence any \(L_X\) with more than one edge gives a strictly smaller ratio than a single-edge perturbation.  The one-edge value is \(\rho_r\), as in Proposition~\ref{prop:near-extremal}.
\end{proof}

\medskip\noindent\textbf{Explicit defect subtractions.}

\begin{corollary}\label{cor:selected-families}
Let \(G\) be a finite simple graph of even order and with no isolated vertices.  Let \(\mathcal{S}\) be any set of distinct nonempty odd-cycle families in \(G\).  Then
\[
        \perfmat(G)\le
        \left(
        U(G)^2-
        \sum_{\Ccal\in\mathcal{S}}
        2^{|\Ccal|}\perfmat(G-V(\Ccal))^2
        \right)^{1/2}.
\]
\end{corollary}

\begin{proof}
The defect \(\delta_{\mathrm{odd}}(G)\) is a sum of nonnegative terms, one for each nonempty odd-cycle family.  Hence
\[
        \delta_{\mathrm{odd}}(G)\ge
        \sum_{\Ccal\in\mathcal{S}}
        2^{|\Ccal|}\perfmat(G-V(\Ccal))^2.
\]
Substituting this lower bound into Corollary~\ref{cor:AF-defect} proves the claim.
\end{proof}

\begin{corollary}[Triangle-pair subtraction]\label{cor:triangle-pairs}
Let \(G\) be a finite simple graph of even order and with no isolated vertices.  Let \(\mathcal{T}\) be any set of unordered pairs \(\{T_1,T_2\}\), where \(T_1\) and \(T_2\) are vertex-disjoint triangles of \(G\).  Then
\[
        \perfmat(G)\le
        \left(
        U(G)^2-4\sum_{\{T_1,T_2\}\in\mathcal{T}}
        \perfmat(G-V(T_1)-V(T_2))^2
        \right)^{1/2}.
\]
In particular, each pair of vertex-disjoint triangles whose deletion leaves at least \(r\) perfect matchings contributes at least \(4r^2\) to the defect.
\end{corollary}

\begin{proof}
Each pair \(\{T_1,T_2\}\in\mathcal{T}\) is an odd-cycle family of size \(2\).  Applying Corollary~\ref{cor:selected-families} to the corresponding set of odd-cycle families gives the stated inequality.
\end{proof}

\begin{corollary}\label{cor:one-family}
Let \(G\) be a finite simple graph of even order and with no isolated vertices.  Suppose \(G\) contains a nonempty odd-cycle family \(\Ccal\) such that
\[
        \perfmat(G-V(\Ccal))\ge r
\]
for some integer \(r\ge 1\).  Then
\[
        \perfmat(G)\le \left(U(G)^2-2^{|\Ccal|}r^2\right)^{1/2}.
\]
In particular, if \(\Ccal\) covers all vertices of \(G\), then
\[
        \perfmat(G)\le \left(U(G)^2-2^{|\Ccal|}\right)^{1/2}.
\]
\end{corollary}

\begin{proof}
Apply Corollary~\ref{cor:selected-families} with \(\mathcal{S}=\{\Ccal\}\).  If \(\Ccal\) covers all vertices, then \(G-V(\Ccal)\) is the empty graph, which has exactly one perfect matching.
\end{proof}

\begin{example}\label{ex:prism}
Let \(G=C_3\square K_2\), the triangular prism, where \(\square\) denotes the Cartesian product of graphs.  Then \(G\) is \(3\)-regular on six vertices, so
\[
        U(G)=\prod_{v\in V(G)} (3!)^{1/(2\cdot 3)}=6.
\]
The graph has \(4\) perfect matchings: the vertical matching, and the three matchings obtained by choosing one vertical edge together with the opposite edge in each triangular face.  Hence \(\perfmat(G)^2=16\).

The two triangular faces form an odd-cycle family \(\Ccal\) with \(|\Ccal|=2\) and \(G-V(\Ccal)=\varnothing\).  Therefore Corollary~\ref{cor:one-family} gives
\[
        \perfmat(G)\le \sqrt{6^2-2^2}=\sqrt{32}<6.
\]
Thus the defect estimate is strictly sharper than the direct Alon--Friedland estimate for this graph.  The example is meant only to illustrate the subtraction mechanism; it is not meant to suggest that partial defect information usually gives sharp numerical bounds.  The complete-graph calculation below is the main family in which the defect has asymptotic force.
\end{example}

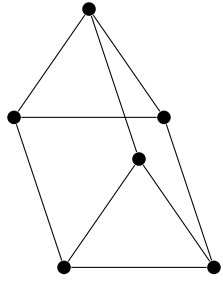
\begin{figure}[ht]
\centering
\begin{tikzpicture}[scale=1.1, every node/.style={circle,fill=black,inner sep=1.8pt}]
\node (a1) at (0,1.8) {};
\node (a2) at (1.8,1.8) {};
\node (a3) at (0.9,3.1) {};

\node (b1) at (0.6,0) {};
\node (b2) at (2.4,0) {};
\node (b3) at (1.5,1.3) {};

\draw (a1)--(a2)--(a3)--(a1);
\draw (b1)--(b2)--(b3)--(b1);
\draw (a1)--(b1);
\draw (a2)--(b2);
\draw (a3)--(b3);
\end{tikzpicture}
\caption{The triangular prism \(C_3\square K_2\).  Its two triangular faces form an odd-cycle family used in Example~\ref{ex:prism}.}
\label{fig:prism}
\end{figure}

\begin{example}[A limitation: \(K_4\)]\label{ex:K4}
For \(G=K_4\), one has \(\perfmat(K_4)=3\).  Also
\[
        U(K_4)=\prod_{v\in V(K_4)}(3!)^{1/(2\cdot 3)}=6^{2/3}.
\]
Every nonempty odd-cycle family in \(K_4\) consists of a single triangle, and deleting that triangle leaves one isolated vertex.  Hence every nonempty odd-cycle-family term vanishes, so \(\delta_{\mathrm{odd}}(K_4)=0\).  Thus the odd-cycle defect does not detect all nonbipartite behavior; it detects only odd-cycle families whose deletion leaves perfect matchings.  In particular, the conjecturally extremal graph in Problem~\ref{prob:sharp-uniform-gap} for small maximum degree is invisible to the defect, so the defect alone cannot resolve the sharp uniform-gap problem.
\end{example}

\begin{table}[ht]
\centering
\small
\renewcommand{\arraystretch}{1.2}
\caption{Sample behavior of the odd-cycle defect.}
\label{tab:samples}
\begin{tabular}{|l|c|c|p{5.4cm}|}
\hline
\textbf{Graph \(G\)} & \(\boldsymbol{\perfmat(G)}\) & \(\boldsymbol{U(G)}\) & \textbf{Defect behavior} \\
\hline
\(C_4\) & \(2\) & \(2\) &
Bipartite case.  No nonempty odd-cycle family exists, so \(\delta_{\mathrm{odd}}(G)=0\). \\
\hline
\(K_4\) & \(3\) & \(6^{2/3}\) &
Nonbipartite, but every triangle deletion leaves an isolated vertex.  Thus \(\delta_{\mathrm{odd}}(K_4)=0\). \\
\hline
\(C_3\square K_2\) & \(4\) & \(6\) &
The two triangular faces give contribution \(2^2=4\), hence
\(\perfmat(G)\le \sqrt{36-4}=\sqrt{32}<6\). \\
\hline
\end{tabular}
\end{table}

\begin{proposition}\label{prop:finite-component-gap}
Fix integers \(\Delta\ge 2\) and \(M\ge 2\).  Then there exists a constant \(c(\Delta,M)>0\) with the following property.  Let \(G\) be a finite simple graph of even order, with no isolated vertices and maximum degree at most \(\Delta\).  Suppose that \(G\) has a connected component \(H\) of even order such that \(|V(H)|\le 2M\), and \(H\not\cong K_{r,r}\) for every \(1\le r\le \Delta\).  Then
\[
        \perfmat(G)\le (1-c(\Delta,M))U(G).
\]
\end{proposition}

\begin{proof}
Write
\[
        G=\bigsqcup_t G_t
\]
as a disjoint union of connected components.  Both \(\perfmat\) and \(U\) factor over connected components, so
\[
        \frac{\perfmat(G)}{U(G)}
        =
        \prod_t \frac{\perfmat(G_t)}{U(G_t)},
\]
with the convention that a component with no perfect matching contributes the factor \(0\).

There are only finitely many connected simple graphs \(Q\) satisfying all of the following conditions: \(|V(Q)|\le 2M\), \(|V(Q)|\) is even, \(Q\) has no isolated vertices, \(\Delta(Q)\le \Delta\), and \(Q\not\cong K_{r,r}\) for every \(1\le r\le \Delta\).  For each such \(Q\), the equality characterization in \eqref{eq:AF} yields
\[
        \frac{\perfmat(Q)}{U(Q)}<1.
\]
Let \(\rho(\Delta,M)\) be the maximum of these finitely many ratios.  Then \(\rho(\Delta,M)<1\).  Put
\[
        c(\Delta,M)=1-\rho(\Delta,M)>0.
\]
The distinguished component \(H\) contributes at most \(\rho(\Delta,M)\), while every other component contributes at most \(1\) by \eqref{eq:AF}.  The result follows.
\end{proof}

\begin{remark}\label{rem:bregman-slack}
The identity \eqref{eq:per-defect} is exact.  Therefore the remaining gap between \(\perfmat(G)^2+\delta_{\mathrm{odd}}(G)\) and \(U(G)^2\) comes entirely from the Bregman--Minc inequality applied to \(A(G)\).  Thus the odd-cycle defect and the Bregman--Minc slack measure different sources of non-sharpness in the Alon--Friedland estimate.
\end{remark}

\medskip\noindent\textbf{Complete graphs and a derangement identity.}

Let \(D_m\) denote the number of derangements of \([m]\).  Equivalently,
\[
        D_m=\per(J_m-I_m),
\]
where \(J_m\) is the all-one matrix and \(I_m\) is the identity matrix.  The cycle-length enumeration used below is classical; see, for example, the exponential formula in \cite[Ch.~II]{FlajoletSedgewick} and the discussion of permutations by cycle structure in \cite[Ch.~3]{Bona}.  The OEIS entry \cite{OEISA001818} records the sequence \(((2n-1)!!)^2\), equivalently the number of permutations in \(S_{2n}\) all of whose cycles have even length, or all of whose cycles have odd length.

\begin{proposition}\label{prop:derangement}
Let \(g(j)\) be the number of permutations of \([j]\) all of whose cycles have odd length at least \(3\).  Then
\begin{equation}\label{eq:g-egf}
        \sum_{j\ge 0} g(j)\frac{x^j}{j!}
        =e^{-x}\sqrt{\frac{1+x}{1-x}}.
\end{equation}
Moreover, for every \(n\ge 1\), with the convention \((-1)!!=1\),
\begin{equation}\label{eq:derangement-identity}
        D_{2n}=((2n-1)!!)^2+
        \sum_{k=2}^{n}\binom{2n}{2k}g(2k)((2n-2k-1)!!)^2.
\end{equation}
\end{proposition}

\begin{proof}
Since \(A(K_{2n})=J_{2n}-I_{2n}\), the permanent \(\per(A(K_{2n}))\) counts fixed-point-free permutations of \([2n]\).  Hence
\[
        \Pi(K_{2n})=D_{2n}.
\]
Also
\[
        p(K_{2n})=(2n-1)!!.
\]
Apply Theorem~\ref{thm:odd-expansion} to \(K_{2n}\).  An odd-cycle family covering \(2k\) vertices leaves a complete graph \(K_{2n-2k}\), which has
\[
        p(K_{2n-2k})^2=((2n-2k-1)!!)^2
\]
perfect-matching pairs, with the convention that this quantity is \(1\) when \(k=n\).

For a fixed \(2k\)-element vertex set, summing the weight \(2^{|\Ccal|}\) over all undirected odd-cycle families \(\Ccal\) covering that set counts the ways to orient all cycles in the family.  This is exactly the number of permutations of the fixed set whose cycles all have odd length at least \(3\), namely \(g(2k)\).  Choosing the \(2k\)-element set gives the factor \(\binom{2n}{2k}\).  The case \(k=0\) gives the term \(((2n-1)!!)^2\).  The \(k=1\) term is absent because \(g(2)=0\): there is no permutation of a two-element set whose cycles all have odd length at least \(3\).  At the other endpoint, \(k=n\) is included, and the convention \((-1)!!=1\) records that the leftover graph is empty.  This proves \eqref{eq:derangement-identity}.

It remains to prove \eqref{eq:g-egf}.  A single directed cycle of length \(\ell\) on a labelled \(\ell\)-set contributes \(x^\ell/\ell\) to the exponential generating function.  The allowed cycle lengths are odd integers \(\ell\ge 3\), so the cycle generating function is
\[
        \sum_{\substack{\ell\ge 3\\ \ell\text{ odd}}}\frac{x^\ell}{\ell}
        =\operatorname{arctanh}(x)-x.
\]
By the exponential formula for labelled permutations as sets of cycles \cite[Ch.~II]{FlajoletSedgewick},
\[
        \sum_{j\ge 0} g(j)\frac{x^j}{j!}
        =\exp(\operatorname{arctanh}(x)-x)
        =e^{-x}\sqrt{\frac{1+x}{1-x}}.
\]
\end{proof}

\begin{example}\label{ex:complete-small}
For \(n=3\), one has \(g(6)=40\).  Thus
\[
        \delta_{\mathrm{odd}}(K_6)=\binom{6}{6}40=40,
        \qquad
        D_6=225+40=265.
\]
For \(n=4\), one has \(g(6)=40\) and \(g(8)=2688\), so
\[
        \delta_{\mathrm{odd}}(K_8)=\binom{8}{6}40+\binom{8}{8}2688=1120+2688=3808,
\]
and hence
\[
        D_8=11025+3808=14833.
\]
\end{example}

\begin{corollary}\label{cor:complete-strict}
For every \(n\ge 3\),
\[
        \perfmat(K_{2n})<U(K_{2n}).
\]
More precisely,
\[
        \delta_{\mathrm{odd}}(K_{2n})=D_{2n}-((2n-1)!!)^2>0.
\]
\end{corollary}

\begin{proof}
For \(n\ge 3\), choose two disjoint triangles in \(K_{2n}\).  They form an odd-cycle family \(\Ccal\), and the complement is \(K_{2n-6}\), which has a perfect matching.  Therefore \(\delta_{\mathrm{odd}}(K_{2n})>0\).  Corollary~\ref{cor:AF-defect} gives the strict inequality.  The formula for the defect follows from \eqref{eq:per-defect} and \(\per(A(K_{2n}))=D_{2n}\).
\end{proof}

\medskip\noindent\textbf{Asymptotic consequences for complete graphs.}

The next result shows that, for complete graphs, the odd-cycle defect accounts asymptotically for almost all of the Bregman--Minc target \(U(G)^2\).

\begin{proposition}\label{prop:complete-asymptotic}
As \(n\to\infty\),
\[
        \frac{\delta_{\mathrm{odd}}(K_{2n})}{U(K_{2n})^2}
        =1-\frac{e}{\sqrt{\pi n}}+o(n^{-1/2}).
\]
In particular,
\[
        \frac{\delta_{\mathrm{odd}}(K_{2n})}{U(K_{2n})^2}\longrightarrow 1.
\]
\end{proposition}

\begin{proof}
For \(K_{2n}\), every vertex has degree \(2n-1\), so
\[
        U(K_{2n})^2=((2n-1)!)^{2n/(2n-1)}.
\]
By \eqref{eq:per-defect},
\[
        \delta_{\mathrm{odd}}(K_{2n})=D_{2n}-((2n-1)!!)^2.
\]
We estimate the two terms after division by \(U(K_{2n})^2\).

First, \(D_{2n}=(2n)!/e+O(1)\).  Put \(m=2n-1\).  Since
\[
        U(K_{2n})^2=(m!)^{(m+1)/m}
        \quad\text{and}\quad
        D_{2n}\sim \frac{(m+1)!}{e},
\]
we have
\[
        \log\frac{U(K_{2n})^2}{D_{2n}}
        =\frac{1}{m}\log(m!)-\log(m+1)+1+o(1/m).
\]
Stirling's formula gives
\[
        \frac{1}{m}\log(m!)
        =\log m-1+\frac{\log(2\pi m)}{2m}+O(m^{-2}),
\]
and hence
\[
        \log\frac{U(K_{2n})^2}{D_{2n}}
        =\log\frac{m}{m+1}+\frac{\log(2\pi m)}{2m}+O(m^{-2})
        =O\left(\frac{\log n}{n}\right).
\]
Therefore
\begin{equation}\label{eq:D-over-U}
        \frac{D_{2n}}{U(K_{2n})^2}
        =1+O\left(\frac{\log n}{n}\right).
\end{equation}

Second,
\[
        (2n-1)!!=\frac{(2n)!}{2^n n!}.
\]
Since \(D_{2n}\sim (2n)!/e\) and \(\binom{2n}{n}\sim 4^n/\sqrt{\pi n}\), we get
\[
        \frac{((2n-1)!!)^2}{D_{2n}}
        \sim e\frac{\binom{2n}{n}}{4^n}
        \sim \frac{e}{\sqrt{\pi n}}.
\]
Combining this with \eqref{eq:D-over-U} gives
\[
        \frac{((2n-1)!!)^2}{U(K_{2n})^2}
        =\frac{e}{\sqrt{\pi n}}+o(n^{-1/2}).
\]
Finally,
\[
        \frac{\delta_{\mathrm{odd}}(K_{2n})}{U(K_{2n})^2}
        =\frac{D_{2n}}{U(K_{2n})^2}
        -\frac{((2n-1)!!)^2}{U(K_{2n})^2}.
\]
Since \((\log n)/n=o(n^{-1/2})\), the asserted asymptotic follows.
\end{proof}

\begin{corollary}\label{cor:p-over-U-complete}
As \(n\to\infty\),
\[
        \frac{\perfmat(K_{2n})}{U(K_{2n})}
        \sim \sqrt{e}\,(\pi n)^{-1/4}.
\]
In particular, \(\perfmat(K_{2n})/U(K_{2n})\to 0\).
\end{corollary}

\begin{proof}
This follows by taking square roots of the estimate
\[
        \frac{((2n-1)!!)^2}{U(K_{2n})^2}
        \sim \frac{e}{\sqrt{\pi n}},
\]
which was proved in Proposition~\ref{prop:complete-asymptotic}.
\end{proof}


\begin{thebibliography}{99}

\bibitem{AF}
N. Alon and S. Friedland,
The maximum number of perfect matchings in graphs with a given degree sequence,
\emph{Electron. J. Combin.} \textbf{15} (2008), no.~1, Note N13, 2 pp.

\bibitem{Bregman}
L. M. Bregman,
Some properties of nonnegative matrices and their permanents,
\emph{Soviet Math. Dokl.} \textbf{14} (1973), 945--949.

\bibitem{Bona}
M. B\'ona,
\emph{Combinatorics of Permutations}, second edition,
Discrete Mathematics and its Applications, CRC Press, Boca Raton, FL, 2012.

\bibitem{FlajoletSedgewick}
P. Flajolet and R. Sedgewick,
\emph{Analytic Combinatorics},
Cambridge University Press, Cambridge, 2009.

\bibitem{Galvin}
D. Galvin,
Three tutorial lectures on entropy and counting,
arXiv:1406.7872, 2014.

\bibitem{LovaszPlummer}
L. Lov\'asz and M. D. Plummer,
\emph{Matching Theory},
AMS Chelsea Publishing, Providence, RI, 2009.

\bibitem{Minc}
H. Minc,
Upper bounds for permanents of \((0,1)\)-matrices,
\emph{Bull. Amer. Math. Soc.} \textbf{69} (1963), 789--791.

\bibitem{OEISA001818}
The OEIS Foundation Inc.,
The On-Line Encyclopedia of Integer Sequences, Sequence A001818,
published electronically at \url{https://oeis.org/A001818}, accessed July 2026.

\bibitem{Radhakrishnan}
J. Radhakrishnan,
An entropy proof of Bregman's theorem,
\emph{J. Combin. Theory Ser. A} \textbf{77} (1997), 161--164.

\end{thebibliography}
\end{document}